\documentclass[12pt]{amsart}

\usepackage{amsmath,amssymb,amscd,latexsym,eufrak,oldgerm
}

\usepackage{graphicx}

\newcommand{\svskip}{\vspace{3mm}}

\newcommand{\C}{{\Bbb C}}
\newcommand{\CC}{{\Bbb C}}
\newcommand{\R}{{\Bbb R}}

\newcommand{\Z}{{\Bbb Z}}

\newcommand{\Q}{{\Bbb Q}}
\newcommand{\U}{{\Bbb U}}

\newtheorem{thm}{Theorem}[section]
\newtheorem{lem}[thm]{Lemma}
\newtheorem{prop}[thm]{Proposition}
\newtheorem{cor}[thm]{Corollary}
\newtheorem{propdefn}[thm]{Proposition-Definition}
\newtheorem{remark}[thm]{{\sc Remark}}

\newtheorem{example}[thm]{{\sc Example}}

\begin{document}

\title{Multivariable Hodge theoretical invariants of germs 
of plane curves I} 

\author{Pierrette Cassou-Nogu\`es}
\address{Department of Mathematics\\
Universite Bordeaux 1}
\email{Pierrette.Cassou-nogues@math.u-bordeaux1.fr}

\author{Anatoly Libgober}
\address{Department of Mathematics \\
University of Illinois, Chicago, IL 60607}
\email{libgober@math.uic.edu}
\date{} 
\begin{abstract}
We describe methods for calculation of polytopes of quasiadjunction 
for plane curve singularities which are invariants 
giving a Hodge theoretical 
refinements of the zero sets of multivariable Alexander polynomials.
In particular we identify some hyperplanes on which all polynomials
in multivariable Bernstein ideal vanish.
\end{abstract}
\maketitle
\section{Introduction}

The purpose of this paper is to study Hodge theoretical 
invariants of local systems on the complements to 
germs of plane curve singularities. 
These invariants, called the faces of quasiadjunction,  
yield a refinement for the multivariable Alexander polynomial
of a link of isolated singularity 
or, more precisely, for the characteristic varieties associated 
with the homology of universal abelian covers of the complements
to a germ of plane curve. 
We develop algorithmic methods for calculation of these
Hodge theoretical invariants in terms of combinatorics of defining
equations of germs and to use them to study distributions
of polytopes for all germs of plane curves.

While a multivariable Alexander polynomial is a product 
of factors of the form $t_1^{m_1}...t_r^{m_r}-\omega$ where 
$\omega=e^{2 \pi \sqrt {-1} \alpha}$ 
is a root of unity (i.e. ${\alpha}, \alpha \in \bf Q$), 
the faces of quasiadjunction 
are subsets of the union of hyperplanes of the form 
$\sum m_ix_i=\alpha+k, \ k \in \Z$ and are related to the 
zero set of the Alexander polynomial via the exponential map. 
Faces of quasiadjunction 
for links of plane curve singularities were defined in 
\cite{charvar}, \cite{Hodge} using adjunction conditions 
for the abelian covers of germs. In these papers they 
were related to the cohomology of the finite order 
local systems of the complements.
Here we show that faces of quasiadjunction can be  
 defined entirely in terms of Mixed Hodge structure, described 
in \cite{nonvanloci}, 
on the cohomology of 
local systems (non necessarily of finite order)
on the complement to the germ: they
 parametrize the unitary local systems with the dimension 
of a Hodge component having a fixed value.
In particular the families of (unitary) local systems 
on the complement to a link with fixed Hodge data 
have a linear structure and form polytopes of various dimensions
in the universal covering of 
the space of unitary local systems. This description of the 
Mixed Hodge structure on the local systems in terms of log-resolution 
of the germ provides
first approach to calculations of faces of quasiadjunction.

The second approach discussed in this paper is based on 
the relation of faces of quasiadjunction with the spectrum of 
non reduced singularities (cf.\ref{localsystemspectrum}).
Below we show that ``most'' local systems with non vanishing 
Hodge numbers correspond to   
eigenvalues of the (semi-simple part of) the monodromy 
acting on graded vector space associated with limit Hodge 
filtration on the cohomology of the Milnor fiber of 
the non-reduced singularity of 
the form $f_1^{a_1} \cdot .... \cdot f_r^{a_r}$.  
Another technical result which makes calculations more 
effective shows that while the definition of faces of quasiadjunction
describes them in terms of  multiplicities 
of pull backs of defining equations and differentials 
along the exceptional curves, this data corresponding 
only to curves having at least three intersection points 
with other exceptional curves suffices. This can be viewed 
as an extension of what was known 
previously in uni-branched case and in the case of  
multivariable Alexander polynomials (cf. also
related result which concerns the log-canonical thresholds 
\cite{smith}).

We start with review in section \ref{sectiononideals}
of the definitions of the polytopes and ideals
of quasiadjunction, multiplier ideals and their basic properties. In section 
\ref{faces} we show the
connection with the Arnold-Steenbrink spectrum of non reduced 
singularities. As an application we show in section \ref{sectionbernstein}
that all multivariable Bernstein polynomials, i.e. the 
polynomials $b(s_1,...,s_r)$ for which there exist a differential 
operator $P$ such that $b(s_1,...,s_r)f_1^{s_1} \cdot ...\cdot f_r^{s_r}
=Pf_1^{s_1+1} \cdot....\cdot f_r^{s_r+1}$ are vanishing on 
hyperplanes containing codimension one faces of quasiadjunction.

The part II of this paper (\cite{PartII}) develops methods 
for calculations of the ideals and 
polytopes of quasiadjunction from defining equations of the 
germs. The main technical tool 
are the Newton trees introduced by the first named author.
The Newton trees are very close to Eisenbud and Neumann diagrams.
Moreover we introduce a partial resolution of singularity of the 
germ of a pair $(\C^2,D)$ which we call Newton space and which
is  a toroidal morphism $\tilde C^2 \rightarrow \C^2$  of the 
space $\C^2$ having only quotient singularities.  
The attractive feature of the Newton space is that it 
has canonical construction in terms of Newton tree
and does not involve choice (while classically used in 
this context resolutions do).
A byproduct of this construction is 
the fact that only the divisors in the 
resolution which intersect at least three other
divisors appear in the computation of the polytopes of quasiadjunction. 
A subsequent paper also describes application of these methods 
to a study of the polytopes of quasiadjunction and spectrum 
at infinity which serve as refinements of the studied 
earlier the Alexander polynomials at infinity.
Part II (cf. \cite{PartII}) also includes 
several explicit examples of calculation 
of polytopes and ideals of quasiadjunction.

First name author wants to thank Radcliffe Institute at Harvard 
University where the work on this project began. She was supported 
by MTM2007-67908-CO2-01 and MTM2007-67908-CO2-02.
The second named author 
wants to thank University of Bordeaux, where work on this 
project continued, for warm hospitality and support. He 
was supported MTM2007-67908-CO2-01 and by NSF grant.

\section{Ideals of quasiadjunction and multiplier ideals}\label{sectiononideals}

The ideals and polytopes of quasiadjunction, which are the main 
objects of this paper, are closely related 
to other invariants, i.e. the multiplier ideals,
 which recently attracted much attention. 
We shall briefly review 
their main properties relevant to this work. To describe them, 
let $X$ be a smooth complex affine variety of dimension $n$ and 
$\bf{a}$ an ideal in $\C[X]$.
A log-resolution of 
$\bf{a}$ is a proper birational map 
$\mu: Y \longrightarrow X$ whose exceptional locus is a divisor $E$ such that
\begin{enumerate}
\item $Y$ is non singular
\item  
$\bf{a} \mathcal{O}_Y=\mathcal{O}_Y (-F)$ with $F=\sum r_iE_i$ 
an effective divisor,
\item $F+E$ has simple normal crossing support.
\end{enumerate} 

Consider the relative 
canonical divisor $\kappa _{Y/X}=\kappa _Y -\mu ^* \kappa _X$, and
write $\kappa_{Y/X}=\sum b_iE_i$.

For every rational number $c\geq 0$ the multiplier ideal of the ideal 
$\bf{a}$ with exponent
$c$ is 
$$\Im (\bf{a}^c)=\{h\in \C[X] \vert \text{ord} _{E_i}(\mu ^*h)\geq [cr_i]-b_i, \forall i\}$$
where $[\ ]$ is the round down operation. The definition is independent of the chosen resolution.

More generally one can consider the jumping numbers of 
$\bf{a}=(f)$ i.e. 
one can prove that there exists  an increasing sequence  of rational numbers 
$$0 =\xi _0<\xi _1<\xi _2<\cdots $$
such that $\Im (\bf{a}^c)$ is constant on $\xi _i\leq c <\xi _{i+1}$, and 
$\Im (\bf{a}^{\xi _i})\neq \Im (\bf{a}^{\xi _{i+1}})$ .

The numbers $\xi _i$ are called the {\it jumping numbers} of $(f)$. If $f$ has isolated singularity at the origin.
It follows from \cite{loser}
 that if $\alpha \in (0,1]$, then $\alpha $ belongs to the 
Arnold-Steenbrink-Hodge spectrum of $f$  if and only if $\alpha$ is a jumping number. 

Moreover, Ein, Lazarfeld, Smith and Varolin (cf.\cite{ELSV})
have shown that if $\xi$ is a jumping number of $f$ in $(0,1]$, then $b_f (-\xi)= 0$ where $b_f(s)$ is the Bernstein polynomial of $f$.

In this paper we are interested in more general multiplier ideals, called the {\it mixed multiplier ideals}.
Let $\bf{a}_1, \cdots , \bf{a}_t$ be ideals in $\C[X]$ and $c_1,\cdots, c_t$ be positive numbers and consider a log resolution of $\bf{a}_1\cdots  
\bf{a}_t$. We write 
$\bf{a}_i\mathcal{O}_Y=\mathcal{O}_Y(-F_i)$ with $F_i=\sum r_{i,j}E_j$, then
$$\Im (\bf{a}_1^{c_1}\cdots \bf{a}_t^{c_t})=\{h\in \C[X] \vert \text{ord} _{E_j}(\mu ^*h)\geq [c_1r_{1,j}+\cdots+c_tr_{t,j}]-b_j, \forall j\}$$ 
Again the definition does not depend on the chosen resolution. We will restrict ourselves to the case where the ideals $\bf{a}_i=(f_i)$ are principal.

It turns out that the mixed multiplier ideals coincide with the 
ideals of quasiadjunction (cf. \cite{Hodge}).
 For these ideals there is a notion of ``jumping numbers'' 
which in this case are the {\it faces of quasiadjunction} referred to earlier. 
There is also a notion of log canonical threshold. 
In this paper we focus on the case where $n=2$. 

Now let us review the definition of ideals and polytopes of quasiadjunction.
Let $B$ be a small ball about the origin in $\C^2$ and let $\mathcal{C}$ be a germ of a plane curve having at $0$ a singularity with $r$ branches. 
Let $f_1(x,y)\cdots f_r(x,y)=0$ be a local equation of this curve 
(each $f_i$ is assumed to be irreducible). 
An abelian cover of type $(m_1,\cdots,m_r)$ of  
$\partial B$ is the link of complete intersection surface singularity: 
\begin{equation}\label{zandf}
 V_{m_1,\cdots,m_r}: z_1^{m_1}=f_1(x,y),\cdots, z_r^{m_r}=f_r(x,y)
\end{equation}
The covering map is given by $p:(z_1,\cdots, z_r,x,y)\to (x,y)$.

An {\it ideal of quasiadjunction} of type $(j_1,\cdots, j_r\vert m_1,\cdots, m_r)$ is the ideal in the local ring of the singularity of $\mathcal{C}$ consisting of germs $\phi$ such that the $2$-form:
\begin{equation}\label{phiform}
\omega _{\phi}=\frac {\phi z_1^{j_1}\cdots z_r^{j_r}dx \wedge dy}{z_1^{m_1-1}\cdots 
z_r^{m_r-1}}
\end{equation}
extends to a holomorphic form on a resolution of the singularity of $V_{m_1,\cdots,m_r}$. 

An {\it ideal of log-quasiadjunction} (resp. {\it ideal of weight one}
\par \noindent {\it log-quasiadjunction}) 
of type $(j_1,\cdots, j_r\vert m_1,\cdots, m_r)$
is the ideal in the same local ring consisting of 
germs $\phi$ such that the $2$-form $\omega _{\phi}$ extends 
to a log-form (resp. weight one log-form)
 on a resolution of the same singularity. 
We shall denote the ideal of quasiadjunction, 
(resp. log-quasiadjunction, resp. weight one 
log-quasiadjunction 
\footnote{recall that a form $\omega$ on a complement 
to a divisor $\cup E_i$ with normal crossings has weight $w$ 
if in a local coordinates near a point where $E_i$ is given by 
$z_i=0$ one has $\omega=f \cdot {dz_{i_1} \over z_{i_1}} \wedge...
\wedge {dz_{i_w} \over z_{i_w}} \wedge dz_{i_{w+1}} ...$}) corresponding
 to $(j_1,\cdots, j_r\vert m_1,\cdots, m_r)$ 
as $\mathcal{A}(j_1,\cdots, j_r\vert m_1,\cdots, m_r)$  
\newline (resp. $\mathcal{A}''(j_1,\cdots, j_r\vert m_1,\cdots, m_r)$, 
resp. $\mathcal{A}'(j_1,\cdots, j_r\vert m_1,\cdots,m_r)$). We have
 $$\mathcal{A}(j_1,\cdots, j_r\vert m_1,\cdots, m_r)\subseteq $$
 $$\mathcal{A}'(j_1,\cdots, j_r\vert m_1,\cdots, m_r)\subseteq \mathcal{A}''(j_1,\cdots, j_r\vert m_1,\cdots,m_r)$$

\begin{prop}\label{finiteness}
 For any germ with $f_1 \cdot ... \cdot f_r=0$ 
there are finitely many ideals of quasiadjunction 
(and log-quasiadjunction)

\end{prop}

\begin{proof} Let $E_1,...,E_k,... (k \in K)$ be the collection of exceptional 
components of a resolution $\pi$ of this germ. For each map $\Theta: K 
\rightarrow {\bf N}_+$ one can define the ideal 
${\mathcal A}_{\Theta}=\{\phi \in {\mathcal O}_{0,{\bf C}^2} \vert 
{\rm mult}_{E_k}\pi^*(\phi) \ge \Theta(k) \}$.
We shall call $\Theta$ coherent (resp. log-coherent) if there exist 
$(x_1,...,x_r) \ 0 \le x_i <1 $ such that 
$\Theta(k)=\sum a_{k,i}x_i-c_k$ (resp. $\Theta(k)=\sum a_{k,i}x_i-c_k-1$);
other notations are as in \cite{Hodge})
for all $k \in K$.
For each $k$ there is an integer $N_k$ such that for all $k$ and 
and coherent (or log-coherent)
$\Theta$ one has $\Theta(k) \le N_k$. The ideals of quasiadjunction 
are the ideals ${\mathcal A}_{\Theta}$ corresponding to coherent $\Theta$.
There are at most  $\Pi_k N_k$ coherent $\Theta$ and hence the claim 
follows. This proposition reduces calculation of the ideals 
of quasiadjunction to calculation of the ideals corresponding to valuations
and identifying those among them which correspond to 
coherent $\Theta$.

\end{proof}

 It is proved in \cite{Hodge}, Remark 2.6 
that  $\mathcal{A}(j_1,\cdots, j_r\vert m_1,\cdots, m_r)$ is equal to 
$\Im (\bf{a}_1^{c_1}\cdots \bf{a}_r^{c_r})$, where $\bf{a}_i=(f_i)$ and $c_i=1-\frac{j_i+1}{m_i}$ for all $i$.

Let 
$$\mathcal{U}=\{(x_1,\cdots, x_r)\in \R^r, 0\leq x_i<1\}$$
be the unit cube with coordinates corresponding to $f_1,\cdots, f_r$.

\begin{prop}

Let $\mathcal A$ be an ideal of quasiadjunction. There is a
unique polytope ${\mathcal P}({\mathcal A})$ 
open subset in $\mathcal{U}$ such that:
For $(m_1,\cdots, m_r)\in\Z^r$ and 
$(j_1,\cdots, j_r)\in \Z^r$ with $0\leq j_i<m_i, 1\leq i\leq r$
$${\mathcal A} \subseteq {\mathcal A} 
(j_1,...,j_r \vert m_1,...,m_r)\Leftrightarrow (\frac{j_1+1}{m_1},\cdots, 
\frac{j_r+1}{m_r})\in {\mathcal P}({\mathcal A})$$
 \end{prop}
A {\it face of quasiadjunction} is a face of the boundary of the 
polytope ${\mathcal P}({\mathcal A})$. 
It follows that it can be characterized as follows. Let $E_i$ be the 
exceptional curves of an embedded resolution $\pi: \tilde \C^2 \rightarrow 
\C^2$ of 
$f_1 \cdot...\cdot f_r=0$. Let $a_{i,k}=mult_{E_k}\pi^*(f_i)$ be the 
multiplicity of pullback of $f_i$ to $\tilde \C^2$, $c_k=
mult_{E_k}\pi^*(dx \wedge dy)$ and for a germ $\phi \in {\mathcal O}_{O}$,
$e_k(\phi)=mult_{E_k}\pi^*(\phi)$. 
Then the face of quasiadjunction 
containing $\wp=({{j_1+1} \over {m_1}},...,{{j_r+1} \over {m_r}}) 
\in {\mathcal U}$ is the face $\alpha$ of the 
boundary of the set of points satisfying:

\begin{equation}\label{expliciteineq}
    \sum_i a_{i,k}x_i >\sum_i a_{i,k}-e_k(\phi)-c_k-1
\end{equation}
for all $\phi$ in the ideal of quasiadjunction ${\mathcal A}(j_1,..,j_r 
\vert m_1,...,m_r)$ (and such that  $\wp \in \alpha$).
In particular for $(j_1',...,j_r' \vert m_1',...,m_r')$
for which the corresponding point satisfies (\ref{expliciteineq}) 
the form $\omega_{\phi}$ extends over all $E_i$. However for 
$(j_1',...,j_r' \vert m_1',...,m_r')$ on the face itself there exist 
$\phi$ in the ideal of quasiadjunction for which $\omega_{\phi}$
has pole on one of the exceptional curves.

An intrinsic interpretation of the cube $\mathcal U$ is as follows
(cf. \cite{nonvanloci}).
Consider the map $exp: {\mathcal U} \rightarrow {\C^*}^r$ where 
the target is identified with the group of characters of the 
fundamental group $\pi_1(B-\mathcal C)$. Its image consists 
of the unitary characters. The characters of the Galois group 
of the covering map $p: V_{m_1,..,m_r} \rightarrow \C^2$ form 
a finite subgroup in $exp({\mathcal U}) \subset 
Char \pi_1(B-\mathcal C)$. A pull back of the multivalued form 
$\omega_{\phi}$ given by (\ref{phiform}) is single-valued on 
$V_{m_1,...,m_r}-p^{-1}(\mathcal C)$ and is an eigenform 
for the action of the Galois group of $(V_{m_1,...,m_r}-p^{-1}(\mathcal C))/
(B-\mathcal C)$. If loop $\gamma_i$ is
defined by the condition that it has linking number 
with $f_i=0$ (resp. $f_j=0$) in $B$ equal to one (resp. zero) then 
this character is given by:

\begin{equation}\label{chargammai}
   \chi(\gamma_i)=exp (2 \pi \sqrt{-1} {{j_i+1} \over {m_i}})
\end{equation}

The condition that $\phi$ belongs to the ideal  of quasiadjunction 
is equivalent to the condition that $\omega_{\phi}$ represents the
zero cohomology class in $H^1(V_{m_1,..,m_r}-{\mathcal C})$.
The condition that $\phi$ belongs to the ideal log-quasiadjunction 
(resp. weight 2 log-quasiadjunction)
is equivalent to the condition that $\omega_{\phi}$ 
has logarithmic singularities along the exceptional divisor 
of the resolution of singularities of $V_{m_1,...,m_r}$ induced
by the resolution of $f_1 \cdot ...\cdot f_r=0$ (resp. is a weight 2 
logarithmic form).

\section{Faces of quasiadjunction and spectrum}\label{faces}

 In this section we shall describe a relationship between the faces 
of quasiadjunction and the spectrum. 
For this we need to reinterpret the elements of faces of quasiadjunction 
in terms of the cohomology of local systems on $B-{\mathcal C}$.
Recall first the definitions from 
\cite{Hodge} and \cite{nonvanloci}. 
One has $H_1(S^3-L,\Z)=\Z^r$ with generators given by the 
classes of the loops $\gamma_i$ described in just above (\ref{chargammai}). 
In particular 
the local systems on $S^3-L$ are parametrized by 
a torus of characters ${\C^*}^r$. For a fixed rank one local system
$\chi: H^1(S^3-L,\Z) \rightarrow \C^*$, 
the cohomology groups $H^1(S^3-L,\chi)$ support a mixed Hodge 
structure (cf. \cite{nonvanloci}) and we need a description 
of the polytopes of quasiadjunction in terms of this mixed Hodge structure.

\subsection{Mixed Hodge structure on the cohomology 
of a local system on a complement to a link}

Let us describe the mixed Hodge structure on $H^i(S^3-L,\chi)$, 
which is a slight modification of  construction in \cite{nonvanloci}
and method for calculating  the Hodge numbers of 
this mixed Hodge   
structure in this particular case of local systems on the 
complement to a germ of plane curve. 
First we shall select a resolution of singularity of a germ $D$
given by $f_1 \cdot ...\cdot f_r=0$.
Denote the collection of exceptional curves $E_i$ 
(as in  the proof of Prop. \ref{finiteness}).
We shall identify the 3-sphere $S^3$ about the origin 
with the boundary $\partial T(\cup E_i)$ of a regular neighborhood 
$T(\cup E_i)$ of the exceptional set $\cup E_i$. 
If $D^*$ is the proper preimage of $D$ for selected resolution
then $\partial T(\cup E_i) \cap D^*=S^3 \cap D=L$.
In particular one obtains the identification:
\begin{equation}
    S^3-L=(\bigcup_i \partial T(E_i))-D^*
\end{equation}

 The set  $\cup E_i \cup D^*$ has the 
canonical stratification with one dimensional (over $\C$)
strata $E_i^{\circ}$ and zero dimensional strata $E_i \cap E_j$, 
or $E_i \cap D^*$. This induces the decomposition of 
$(\bigcup_i \partial T(E_i))-D^*$ into a union of open 
3-dimensional (over $\R$) 
manifolds each being a circle bundle over $E_i^{\circ}$
which we denote $\partial T(E_i^{\circ})$. 
The intersection of any two such manifolds is a product 
of two punctured disks which we shall denote $\partial_{i,j}$
(it is empty unless $E_i \cap E_j \ne \emptyset)$.
One has the embeddings 
$\phi_{i,j}: \partial_{i,j} \rightarrow \partial T(\cup_k E_k)-D^*$
 each being either of the compositions
 $\partial_{i,j} \rightarrow \partial T(E_i^{\circ})
\rightarrow T(\cup_k E_k)-D^*$ or 
$\partial_{i,j} \rightarrow \partial T(E_j^{\circ})
\rightarrow T(\cup_k E_k)-D^*$. Given a character $\chi \in 
{\rm Char}(\pi_1(S^3-L))$
we denote by $\phi_{i,j}^*(\chi)$ the pull back of $\chi$ to $\partial_{i,j}$.

We shall start by describing the mixed Hodge structure on the cohomology 
of local systems 
on each $\partial T(E_i^{\circ})$ and then show how they yield the 
Hodge numbers of the MHS on the union. 
We can deal only with the case of non trivial
local systems $\chi \in {\rm Char} \pi_1(\partial \cup_iT(E_i)-D^*)
={\rm Char} \pi_1(S^3-L)$  since the case $\chi=1$ dealt with in \cite{DH}.

Let 
\begin{equation}\label{circlefibration}
 \theta: \partial T(E_i^{\circ}) \rightarrow
E_i^{\circ}
\end{equation}
be the canonical circle fibration and let $\gamma_{E_i}$ be the
class of the fiber of the latter in $H_1(\partial T(E_i^{\circ}))$. 

Notice that $H^i(\partial T(E_i^{\circ}),\chi)=0$ for any 
$i$ unless $\chi(\gamma_{E_i})=1$. This follows from 
Leray spectral sequence:
$$E_2^{p,q}=H^p(E_i^{\circ},R^q\theta_*(\chi)) \Rightarrow 
H^{p+q}(\partial T(E_i^{\circ}),\chi)$$
since 
 for a non trivial 
local system $\chi$ on a circle one has $H^i(S^1, \chi)=0$.

From now on we assume that $\chi(\gamma_i)=1$ i.e. that the local 
system, on which cohomology we want to construct a mixed Hodge structure,
is the pullback $\theta^*(\chi)$ of a local system on $E_i^{\circ}$.
Notice that $R^q\theta_*(\C)$ are trivial local systems since 
circle fibrations over $E_i^{\circ}$ are restrictions of the 
circle fibrations over $E_i$ and $\pi_1(E_i)=0$. Leray spectral sequence,
which is a sequence of MHS (cf. \cite{Arapura} for a similar 
statement) has the term $E_2^{p,q}$ as follows:

\begin{equation}
\begin{array}{ccc}
H^0(E_i^{\circ},R^2\theta_*(\CC) \otimes \chi) & 0 
 & 0 \\
H^0(E_i^{\circ},R^1\theta_*(\CC) \otimes \chi) & 
H^1(E_i^{\circ},R^1\theta_*(\CC) \otimes \chi) & 
0 \\
H^0(E_i^{\circ},R^0\theta_*(\CC) \otimes \chi) & 
H^1(E_i^{\circ},R^0\theta_*(\CC) \otimes \chi) 
 &
H^2(E_i^{\circ},R^0\theta_*(\CC) \otimes \chi) 
\end{array}
\end{equation}

It yields hence: 

\begin{equation}
  H^2((\partial T(E_i^{\circ}),\theta^*(\chi)))=
H^1(E_i^{\circ},R^1\theta_*(\CC) \otimes \chi)=
H^1(E_i^{\circ},\chi)(-1)
\end{equation}
(the Tate twist is the contribution of $R^1\theta_*(\CC)$ i.e. 
a shift of weight by 2)
and 
\begin{equation}\label{pullback}
0 \rightarrow H^0(E_i^{\circ},R^1\theta_*(\CC) \otimes \chi) \rightarrow
  H^1(\partial T(E_i^{\circ}),\theta^*(\chi)) \rightarrow 
H^1(E_i^{\circ},\chi) \rightarrow 0
\end{equation}

We can assume that $\chi \in {\rm Char}\pi_1(E_i^{\circ})$ is non trivial 
since the case $\chi=1$ was treated in \cite{Durfee}.
If $\chi \ne 1$ then $H^0(E_i^{\circ}, \chi)=0$ and 
the above exact sequence shows that the  
MHS on the cohomology of $\partial T(E_i^{\circ})$ 
with coefficients in a local system 
coincides with the MHS 
on the cohomology  
of $E_i^{\circ}$:
\begin{equation}\label{oneisomorphism}
H^1(\partial T(E_i^{\circ}),\theta^*(\chi))=H^1(E_i^{\circ},\chi)
\end{equation}

 The MHS 
on the cohomology of local systems on a quasiprojective manifolds 
and relevant MHC were constructed in \cite{Timm}.

To obtain MHS on $(\bigcup_i \partial T(E_i))-D^*$ we first observe 
that 
\begin{equation}\label{removingcomponents}
H^i((\bigcup_i \partial T(E_i))-D^*, \chi)=H^i((\bigcup_i \partial T(E_i'))
-D^*-\bigcup E_i'',\chi)
\end{equation} 
where $E_i'$ are the components such that $\chi(\gamma_{E_i'})=1$ and 
$E_i''$ are the components such that $\chi(\gamma_{E_i''}) \ne 1$.
Indeed, the Mayer Vietoris spectral sequences calculating
the cohomology of the union in both sides of 
(\ref{removingcomponents}) coincide. These spectral sequences are
respectively
\begin{equation}
  E_2^{p,q}=H^p((\partial T(E_i^{\circ})-D^*)^{[q]}, \chi)
\Rightarrow H^{p+q}((\bigcup_i \partial T(E_i))-D^*, \chi)
\end{equation}
and 
\begin{equation}
  E_2^{p,q}=H^p((\partial T({E_i'}^{\circ})-D^*-\bigcup E_i'')^{[q]}, \chi)
\Rightarrow H^{p+q}((\bigcup_i \partial T(E_i'))-D^*-\bigcup E_i'', \chi)
\end{equation}
Here the superscript $\cdot^{[q]}$ denotes a $q+1$-fold intersections.
In our case these intersections are empty if $q=2$ and for $q=1$ they 
are equivalent to the products of punctured disks denoted earlier as 
$\partial_{i,j}$. Notice that we denoted 
by ${E'}^{\circ}$ 
a non singular stratum of stratification of the union 
of components $E'=\bigcup_iE_i'$ for which $\chi(\gamma_{E_i'})=1$ 
and that one obvious identity 
$\bigcup_i E_i^{\circ}=\bigcup {E_i'}^{\circ}-\bigcup E_i''$
yields the identification of the terms $E_2$ (and differentials $d_2$) 
of the above spectral sequences.

On the other hand the local system on 
$(\bigcup_i \partial T(E_i'))
-D^*-\bigcup E_i'')$ (as a consequence of $\chi(\gamma_{E_i'})=1$)
 is the pullback of a local system on 
$(\bigcup_i T(E_i'))
-D^*-\bigcup E_i'')$. Next take a projective surface $Z$ containing 
$\bigcup_i T(E_i)$ and extend the unitary connection corresponding 
to the local system on $(\bigcup_i T(E_i'))
-D^*-\bigcup E_i'')$ from 
$(\bigcup_i T(E_i'))
-D^*-\bigcup E_i'')$ to a meromorphic connection on $Z$.
Then the complement $Z^{\circ}$ to the polar set of this connection 
is a quasiprojective manifold, containing $(\bigcup_i T(E_i'))
-D^*-\bigcup E_i'')$ and supporting unitary local system extending 
the local system $\chi$ on 
$(\bigcup_i \partial T(E_i'))
-D^*-\bigcup E_i'')$. Now one applies the same construction of MHC 
calculating the cohomology of local system $\chi$ as in \cite{Durfee}
so that the following Mayer-Vietoris sequence 
is a sequence of MHSs (we denote by the same letter 
$\chi$ the extensions 
of local system on $(\bigcup_i \partial T(E_i'))
-D^*-\bigcup E_i'')$:

\begin{equation}
 \rightarrow H^i((\bigcup_i \partial T(E_i'))
-D^*-\bigcup E_i'', \chi) \rightarrow 
\end{equation}
$$H^i(Z^{\circ}-\bigcup {E_i^{\circ}}',
 \chi)
\oplus H^i((\bigcup_i T(E_i'))
-D^*-\bigcup E_i'', \chi) \rightarrow H^i(Z^{\circ},\chi) \rightarrow
$$
To obtain the Hodge numbers of 
the mixed Hodge structure on the cohomology of the 
union $H^*(\partial T(\bigcup_i E_i^{\circ}))$ we use the
Mayer-Vietoris spectral sequence which since all triple intersections
are empty becomes the exact sequence:

\begin{equation}\label{mayervietoris}
  \oplus_{i,j} H^0(\partial_{i,j}, \phi_{i,j}^*(\chi))    \rightarrow 
H^1(\partial T(\bigcup_i E_i^{\circ}),\chi) \rightarrow
\oplus_i H^1(\partial T(E_i^{\circ}),\phi_i^*(\chi)) \rightarrow
\end{equation}
$$\rightarrow 
\oplus_{i,j} H^1(\partial_{i,j}, \phi_{i,j}^*(\chi))
$$

Note that $H^1(\C^*)=H^{1,1}(\C^*)=\C$ and hence 
$H^1(\partial_{i,j}, \phi_{i,j}^*(\chi))$ 
(which is non zero only if $\phi_{i,j}^*(\chi)$ is trivial)
has pure Hodge structure of weight $2$.
Moreover, the map 
\begin{equation}
W_2/W_1(\oplus_i H^1(\partial T(E_i^{\circ}),\phi_i^*(\chi))
\rightarrow \oplus_{i,j} H^1(\partial_{i,j}, \phi_{i,j}^*(\chi))
\end{equation}
is injective since the classes of weight 2 correspond to the 
cohomology of the fiber of the map $\partial T(E_i^{\circ}) 
\rightarrow E_i^{\circ}$ i.e. the term 
$H^0(E_i^{\circ},R^1\theta_*(\CC) \otimes \chi)$
in (\ref{pullback}).

Hence we obtain the following exact sequence 
which determines the Hodge numbers of $H^1(S^3-L,\chi)$:
\begin{equation}\label{mayervietoris2}
  \oplus_{i,j} H^0(\partial_{i,j}, \phi_{i,j}^*(\chi))    \rightarrow 
H^1(\partial T(\bigcup_i E_i^{\circ}),\chi) \rightarrow
W_1(\oplus_i H^1(\partial T(E_i^{\circ}),\phi_i^*(\chi))) 
\rightarrow 0
\end{equation}

We shall sum up the above discussion in following

\begin{prop} The isomorphism (\ref{oneisomorphism}) and 
exact sequence (\ref{mayervietoris2})
of MHS determine for the MHS on $H^1(S^3-L,\chi)$
the only non zero 
Hodge numbers $h^{p,q,1}(S^3-L,\chi) \ \ (0 \le p+q \le 1)$ 
uniquely. Moreover, the Hodge numbers $h^{p,q,1}(S^3-L,\chi)$
with $p+q=1$ depend only on the Hodge structure on 
$\oplus_{i \in R} H^1(\partial T(E_i^{\circ}),\phi_i^*(\chi)$)
where $R$ is the set of components of $E$ which are intersected
by at least three other components of exceptional 
divisor while $h^{0,0,1}(S^3-L,\chi)$ depends 
on the cohomology (of dimension zero) of pair 
$(\partial T(\bigcup_{i \in R} E_i),\partial T(\bigcup_{i \in R} E_i
\cap \bigcup_{i \notin R} E_i))$
\end{prop}

The last part which concerns the components with at least three 
intersections with remaining ones follows from the exact sequence
(\ref{mayervietoris2}) since for the component with only two intersections
the curve $E_i^{\circ}$ is homotopy equivalent to a circle
and hence $$H^1(\partial T({E_i'}^{\circ},\phi^*(\chi)))=
H^1(E_i'^{\circ},\chi))=0$$ for non trivial $\chi$ while 
for trivial $\chi$ one has weight two Hodge structure
also does not contribute to the 
cohomology of link.

\begin{remark} In \cite{Hodge} the MHS on the cohomology 
of local systems of finite order was obtained from 
the MHS on the finite abelian branched cover (cf. Prop. 
\ref{quasiadjunctionmhs}).
\end{remark}

\subsection{Depth of Characters}

In this subsection we relate the ideals of quasiadjunction (cf. 
section \ref{sectiononideals}), 
and specifically ${\rm dim}\mathcal{A}''/\mathcal{A}'$ and
 ${\rm dim}\mathcal{A}'/\mathcal{A}$
 to the dimensions of $Gr^p_FGr_q^WH^1(L_{\chi})$. Here
 $L_{\chi}$ is the  
local system corresponding to the character $\chi$ of the 
local fundamental group of the complement to the germ $\mathcal{C}$
taking the value $exp(2 \pi \sqrt{-1} {{j_i+1} \over {m_i}})$ 
on the generator $\gamma_i \in \pi_1(S^3-L)$ 
corresponding to factor $f_i$ (cf. Prop. \ref{quasiadjunctionmhs}).

Recall (cf. \cite{Hodge}) that depth of 
$\chi \in {\rm Char}(\pi_1(S^3-L))$ is 
the dimension ${\rm dim} H^1(S^3-L,\chi)$ of the cohomology 
of the complement with the coefficients in the rank one local system 
corresponding to $\chi$. 
The data of ${\rm dim} Gr^p_FGr_q^WH^1(L_{\chi})$
determines the Hodge theoretical refinement of the depth
(cf. \ref{holomdepth}) which depends only on the data of 
ideals of quasiadjunction. We show how this data 
(determined by the ideals of quasiadjunction) yields the depth 
of a character $\chi$. 

\begin{prop}\label{quasiadjunctionmhs} Let $\phi: {\mathbb Z}^r \rightarrow 
\oplus_{i=1}^{i=r} {\mathbb Z}/m_i$ 
be the surjection and let $\partial V_{m_1,..,m_r}$ be the corresponding 
covering space of $S^3$ branched over the link $L$ which 
we view as the link of singularity (cf. (\ref{zandf})). 

Let $\chi$ be a character belonging to the image of the embedding:
$Char \oplus_{i=1}^{i=r} {\mathbb Z}/m_i \rightarrow Char {\mathbb Z}^r$
induced by $\phi$. Assume that $\chi(\gamma_i) \ne 1$ for any $i=1,...,r$.
Then

(a) $H^1(S^3-L,L_{\chi})=\{v \in H^1(\partial V_{m_1,...,m_r}) \vert g \cdot v
=\chi(g)v \}$
and 

(b) the mixed Hodge structure on $H^1(S^3-L,L_{\chi})$ is induced
by the mixed Hodge structure on $H^1(\partial V_{m_1,...,m_r})$
\end{prop}

\begin{proof} This proposition can be deduced from the following 
two lemmas.

\begin{lem}\label{induction} 
Let $G_i=\oplus_{j=1}^{j=i} 
{\mathbb Z}/{m_j}$ and let $G_r$ acts freely on 
a CW-complex $X$ and $X_i=X/G_i$. Let $\chi \in Hom(G,\C^*)$ 
be such that $\chi_i=\chi \vert_{G_i}$ is non-trivial. 
Let $L_{\chi_i}$ be the local system on $X_i$ induced by $\chi_i$.
Then 
$H^1(X_{i+1},L_{\chi_{i+1}})$ is the eigenspace of  
the generator of $G_{i+1}/G_i$ acting on 
$H^1(X_i,L_{\chi_i})$.
\end{lem}
\begin{lem}\label{afterinduction} 
Let $U_{m_1,...,m_r} \subset \partial V_{m_1,..,m_r}$
 be the complement to the
branching locus of the map 
$$\pi(m_1,...,m_r): \partial V_{m_1,..,m_r} \rightarrow S^3.$$
Let $\chi \in Char \pi_1(S^3-L)$ such that 
$\chi (\gamma_i) \ne 1$. Then  
$H^1(\partial V_{m_1,...,m_r}) \rightarrow H^1(U_{m_1,...,m_r})$
corresponding to the embedding induces an isomorphism 
of $\chi$-eigenspaces.
\end{lem}

Indeed,(a) in the proposition \ref{quasiadjunctionmhs} follows by
 applying \ref{induction} to the subgroups $G_i={\mathbb Z}/{m_1} 
\oplus ... \oplus {\mathbb Z}/{m_i}$ 
of the Galois cover $\partial V_{m_1,...,m_r} \rightarrow 
S^3$. Part (b) follows since the embedding in lemma \ref{afterinduction}
is a morphism of mixed Hodge structures and the deck transformation
preserves MHS on $H^1(U_{m_1,...,m_r})$ since they are induced 
by holomorphic maps.

In order to show \ref{induction} consider the exact sequence 
of chain complexes of $G/G_{i+1}$-modules in which 
we view $C_*(X_i)=C_*(X)^{G_i}$ as $G/G_i$-module
and in which the left map is induced by the multiplication by 
$t-\chi_{i+1}(\delta_{i+1})$ where $\delta_{i+1}$ is a generator 
of the cyclic group $G_{i+1}/G_i$ and $t$ is the generator
of $\C[G_i/G_{i+1}]$:
\begin{equation}
 0 \rightarrow 
     C_*(X_i) \otimes_{\chi_i} \C \rightarrow C_*(X_i) \otimes_{\chi_i} \C
\rightarrow C_* (X_{i+1}) \otimes_{\chi_{i+1}} \C 
\rightarrow 0
\end{equation}
This yields the corresponding cohomology sequence:

\begin{equation}
H^0(X_i,L_{\chi_i}) { \buildrel {t-\chi_i(\delta_i)}
\over\longrightarrow} 
H^0(X_i,L_{\chi_i}) \rightarrow 
\end{equation}
$$H^1(X_{i+1},L_{\chi_{i+1}})
\rightarrow H^1(X_i,L_{\chi_i}) \buildrel {t-\chi_i(\delta_i)} \over 
\rightarrow H^1(X_i,L_{\chi_i})
$$

The map in the upper line is surjective since $\chi_i(\delta_i) \ne 1$ and the 
lemma \ref{induction} follows.

To show lemma \ref{afterinduction} let us consider 
the exact sequence:
\begin{equation}
  H^1(\partial V_{m_1,...,m_r},U_{m_1,...,m_r}) \rightarrow
   H^1(\partial V_{m_1,...,m_r}) \rightarrow H^1(U_{m_1,...,m_r})
\end{equation}
$$ \rightarrow H^2(\partial V_{m_1,...,m_r},U_{m_1,...,m_r})
$$
of the pair 
$(\partial V_{m_1,...,m_r},U_{m_1,...,m_r})$.
Denoting by $T(\pi(m_1,...,m_r)^{-1}(Br))$ 
the regular neighborhood of the ramification 
locus of $\pi(m_1,...,m_r)$ and 
by $\partial \pi(m_1,...,m_r)^{-1}(Br))$ its boundary, 
we have the Galois equivariant identification 
of the relative cohomology:
\begin{equation}
H^i(\partial V_{m_1,...,m_r},U_{m_1,...,m_r})=
\end{equation}
$$
H^{i}(T(\pi(m_1,...,m_r)^{-1}(Br)),
\partial \pi(m_1,...,m_r)^{-1}(Br))
$$
$$=H^{3-i}(
\partial \pi(m_1,...,m_r)^{-1}(Br))
$$
For an element in the last group, corresponding to the
component of the branching locus given by $f_i=0$ 
the action of $\gamma_i$ on it is trivial.
Hence the map:
\begin{equation}
H^1(\partial V_{m_1,...,m_r})_{\chi} \rightarrow H^1(U_{m_1,...,m_r})_{\chi}
\end{equation}
is an isomorphism for a character $\chi$ which 
that $\chi(\gamma_i)\ne 1$ for all $i$.
\end{proof}

\begin{remark} Proposition \ref{quasiadjunctionmhs} is closely related 
to Prop. 4.5 in \cite{INNC}. A different type of relation 
(only for certain Hodge components)
between  cohomology of branched and unbranched covers 
of quasiprojective manifolds is given in \cite{ample} 
\end{remark}

\begin{propdefn}\label{holomdepth}Let $a$ belongs to the fundamental domain
of $H_1(S^3-L,{\bf Z})$ acting on the universal cover of the torus
of unitary characters $Char^u(\pi_1(S^3-L))$. The holomorphic 
weight one depth 
of $\chi =exp(2 \pi i a)$ (denoted $d^h(\chi)$) 
is the dimension of vector space $Gr^1_FGr^W_1H^1(L_{\chi})$.
The weight zero depth of a character $\chi$ is 
the dimension vector space 
$Gr^W_0H^1(L_{\chi})$. We denote it as $d^0(\chi)$.
An integer $w=0,1$ is a weight of $\chi$ if $d_0(\chi) \ne 0$
(resp. $d^h(\chi)\ne 0$). Total depth of $\chi$ is $d(\chi)=
d^h(\chi)+d^0(\chi)$.

The value of the depth (resp. holomorphic depth) 
of a character $exp(2 \pi i a)$
where $a$ belongs to the interior of the face is 
independent of $a$ and is 
called
the  depth (resp. holomorphic depth) 
of the face of quasiadjunction. 
\end{propdefn}

\begin{proof} For the characters of finite order,
 it follows from the identification of the holomorphic 
depths with quotients of ideals of quasiadjunction since 
the ideals of quasiadjunction themselves depend only 
on the face of quasiadjunction. This and the definition 
of MHS for infinite order characters yields the claim in general. 
\end{proof}

\noindent We have the following. 

\begin{prop} If the values of character $\chi$ are $\pm 1$ then the 
depth of a character $\chi$ can be calculated as follows: 
$$d^h(\chi)+d^h(\bar \chi)+d_0(\chi)$$
where $\bar \chi$ is the conjugate character.
Otherwise the depth is equal to:
$$d^h(\chi)+d^h(\bar \chi)+d_0(\chi)+d_0(\bar \chi)$$
\end{prop}

\begin{proof} Since the action of the Galois group 
acting on $Gr^W_1=H^{1,0}\oplus H^{0,1}$ preserves this 
direct sum decomposition one has $(Gr^W_1)_{\chi}=
H^{1,0}_{\chi}+H^{1,0}_{\bar \chi}$. The vector 
space $Gr^W_0$ is defined over $\Z$ and hence the eigenspaces 
corresponding to conjugate characters contribute for
$d_0(\chi)+d_0(\bar \chi)$ unless the characters take 
values $\pm 1$ in which case the dimension of eigenspace
will be $d_0(\chi)$
\end{proof}

\begin{remark} This proposition is closely related to 
Prop. 3.2 in \cite{Hodge}. The contribution of 
characters for which $d_0(\bar \chi) \ne 0$ is missing
there however i.e. the sum in 3.2 should have the term
${\rm dim} {\mathcal A}''_{\Sigma}/{\mathcal A}'_{\Sigma}$ 
for $\chi$ taking values $\ne \pm 1$.
\end{remark}

{\bf Example} (cf.\cite{Hodge}) Consider the singularity $x^r-y^r=0$.
Faces of quasiadjunction are hyperplanes $H_l: x_1+...+x_r=l$
($l=1,...,r-2$) where the holomorphic depth of $H_l$ is 
$r-l-1$. If $\chi=exp(2 \pi i a)$ and $a \in H_l$ then $\bar \chi 
=exp(2 \pi i \bar a)$ where $\bar a \in H_{r-l}$. 
Hence by above proposition the depth of any character
in $t_1 \cdot \cdot \cdot t_r=1$
 is equal $r-l-1+r-(r-l)-1=r-2$. 

\bigskip

\subsection{Spectrum and faces of quasiadjunction}
Next let us consider the relation between Arnold-Steenbrink 
spectrum and the cohomology of local systems. Let us fix the 
zero set of a germ of holomorphic function: 
\begin{equation}\label{nonreducedequation}
f_1^{a_1} \cdot ...\cdot f_r^{a_r}=0
\end{equation}
in a ball $B_{\epsilon}: z_1\bar z_1+z_2 \bar z_2 \le \epsilon$ 
and denote it $Z(f_1,...,f_r)$.
Let $\gamma_i$ be the standard generators of
 $H_1(B_{\epsilon}-Z(f_1,...,f_r),\Z)=\Z^r$ 
used earlier (i.e. given by oriented loops having 
linking number with 
the germ of $f_i$ equal to $+1$).

Let $\U \subset \C^*$ denotes the group of complex numbers having modulus
one. 
For any $\xi \in \U$, let $\chi_{\xi}$ be the local 
system given by $\chi_{\xi}(\gamma_i)=\xi^{a_i}$.
As already mentioned, the 
cohomology $H^1(B_{\epsilon}-Z(f_1,...,f_r),\chi_{\xi}))$
support a mixed Hodge structure (cf. \cite{nonvanloci}). 

Consider the Milnor fiber $M(a_1,...,a_r)$ of the singularity
(\ref{nonreducedequation}) and the mixed Hodge structure on 
$H^1(M(a_1,...,a_r))$ defined by \cite{Steenbrink77}.
Recall that the multiplicity of a spectral pair $(\alpha, w)$ where 
$-\alpha=n-p-\beta, 0 \le \beta <1, \alpha \notin {\bf Z}$
is the dimension of the eigenspace with 
the eigenvalue $exp(2 \pi \sqrt{-1}\alpha)$  of the semi-simple 
part of the monodromy acting on $H^n$ of the Milnor fiber of an 
$n$-dimensional 
singularity. Moreover, the eigenvalues corresponding to the Jordan 
blocks of size $1 \times 1$ (resp. $2 \times 2$) 
of the monodromy appear in the action 
of the semi-simple part of the monodromy 
$T_s$ 
\footnote{i.e. one has $T=T_sT_u$ where $T_s$ is semi-simple and 
$T_u$ is unipotent}
on $Gr^W_1$ (resp. $Gr^W_0$ and $Gr^W_2$ (cf. \cite{Steenbrink77}).

The multiplicity of the eigenvalue of the monodromy is related to 
the cohomology of local systems as follows:

\begin{prop}\label{localsystemspectrum}
 Let $m_{\xi}$ denotes the multiplicity of $\xi$ as 
the eigenvalue of the semi-simple part of the monodromy of singularity 
(\ref{nonreducedequation}) acting on $Gr^0_FH^1(M(a_1,...,a_r))$
 Then $$m_{\xi}=d({\chi_{\xi}})$$
\end{prop}

\begin{proof} We shall use the following Steenbrink-Wang sequence 
of mixed Hodge structure on the semi-stable reduction i.e. 
finite degree covering of the base $\Delta_s^* \rightarrow \Delta^*$ of the map
$F: B_{\epsilon}-Z(f_1,..f_r) \rightarrow \Delta^*$ such that
the pullback of (\ref{nonreducedequation}) 
$(B_{\epsilon}-Z(f_1,..f_r))_{s} \rightarrow \Delta_s^*$
has multiplicity one along components of the exceptional set.
We have the sequence of MHS (cf. \cite{SteenbrinkPeters}, p.275):

\begin{equation}\label{wang}
  H^1((B_{\epsilon}-Z(f_1,..f_r))_{s} \rightarrow 
H^1(M(a_1,...,a_r))
\buildrel log(T_u-I) \over \rightarrow H^1(M(a_1,...,a_r))(-1)
\end{equation}

The left term, which is the cohomology of punctured neighborhood,
on semi-stable reduction of a resolution, 
has the MHS defined in terms of the log-complex associated with exceptional 
divisor (cf. \cite{Steenbrink77}, \cite{SteenbrinkPeters}). The two remaining
terms have the limit Mixed Hodge structure with the weight filtration 
given in terms of the monodromy (cf. \cite{Steenbrink77}).
Both, the cohomology of local system as in proposition and the 
multiplicity of the eigenspace of $T_{s}$ have the same description
in terms of  $H^1((B_{\epsilon}-Z(f_1,..f_r))_{s})$ compatible 
with the MHS. Indeed, the mixed Hodge structure on 
$B_{\epsilon}-Z(f_1,..f_r)_s$ together with action 
of the deck transformation, which is holomorphic and hence
preserves MHS, yields the MHS on $B_{\epsilon}-Z(f_1,...,f_r)$.
The same expression takes place for the relation 
between the MHS on Milnor fiber and on semi-stable reduction.
More precisely, one has the commutative diagram comparing the 
two rows of (\ref{wang}):
\begin{equation}
 \begin{matrix} 0 & \rightarrow & 
H^1((B_{\epsilon}-Z(f_1,..f_r))_s) & \rightarrow & 
H^1(M(a_1,...,a_r))  \cr
  & &  T-\lambda I \ \     \downarrow &  &   T_{s} -\lambda 
I \ \ \downarrow \cr 
0 & \rightarrow & H^1((B_{\epsilon}-Z(f_1,..f_r))_{s}) & \rightarrow  & 
H^1(M(a_1,...,a_r)) 
\end{matrix} 
\end{equation}
where $T_{s}$ is the semi-simple part of the monodromy 
acting on the Milnor fiber and $T$ is the (holomorphic deck transformation).
The cokernel of the left map yields the cohomology of the local system 
corresponding to $\lambda$ and cohomology of the right is 
the $\lambda$-eigenspace of the semi-simple part of the monodromy 
acting on the cohomology of Milnor fiber. Since the horizontal 
maps are the maps of MHS we obtain the claim.
\end{proof}

Now we shall describe the faces of quasiadjunction of 
$f_1 \cdot \cdot \cdot f_r$
in terms of faces of quasiadjunction of functions
$f_1^{a_1} \cdot \cdot \cdot f_r^{a_r}$.  We shall use only
collections $(a_1,...,a_r) \in {\bf Z}^r_+$ for which 
$gcd(a_1,..,a_r)=1$ i.e. for which the Milnor fiber 
is connected. Let us consider the subgroup $\oplus (a_i{\bf Z})
 \subset {\bf Z}^r$ where $a_i{\bf Z}$ is the subgroup of index $a_i$.
If ${\mathcal U}$ is the unit cube in ${\bf R}^r$, which we view as  
the fundamental domain of ${\bf Z}^r$ acting on 
${\bf R}^r$ via translations, then the fundamental domain
of $\oplus (a_i{\bf Z})$ is the union of $\Pi a_i$ unit cubes i.e. 
the subset of ${\bf R}^r$ given by:

$$\{(x_1,....,x_r) \in {\bf R}^r \vert 0 \le x_i \le a_i \}$$

We shall denote this subset of ${\bf R}^r$ by ${\mathcal U}(a_1,...,a_r)$ 
or ${\mathcal U}({\bf a})$. Each face of quasiadjunction via translations
yields $\Pi a_i$ corresponding polytopes in ${\mathcal U}({\bf a})$. 
The collection of translates of a face $F$ in ${\mathcal U}({\bf a})$
we shall denote $F({\bf a})$.   

\bigskip

\begin{prop}The rational numbers $\in (0,1)$ 
belongs to the spectrum of singularity 
$\Pi f_i^{a_i}$ iff $(a_1s,...,a_rs)$ belongs to a translate 
$F' \in F({\bf a}) \subset {\mathcal U}({\bf a})$
for some face of quasiadjunction $F$.
Moreover the multiplicity of $s$ in the 
spectrum is the depth of the face which translate the segment 
$L_{\bf a}=(a_1t,...,a_rt)\ \ (t \in (0,1))$ 
intersects at the point corresponding to $t=s$.

\par Vice versa, a face of quasiadjunction of depth $k$ is a connected
component of closure of points $(b_1,..,b_r) \in {\mathcal U} \cap {\bf Q}$
such that for some $s \in (0,1)$ and $(a_1,..,a_r) \in {\bf Z}^r_{+}
\ \ (gcd(a_i)=1)$
one has $b_i=a_is$ and $s$ is an element of the spectrum of 
singularity $\Pi f_i^{a_i}$ having multiplicity $k$.   

\end{prop}

\bigskip

{\bf Proof.} This proposition follows from the relation between
Betti numbers or ranks of Hodge groups of local systems 
and faces of quasiadjunction.
Recall that that the local system corresponding to the
character $\chi, \chi(\gamma_i)=\xi_i$ has non vanishing 
cohomology iff $(\xi_1,...,\xi_r)$ belongs
to the characteristic variety (cf. remark \ref{alexanderpolynomial}
below). Similarly, $(s_1,..,s_r), 0 < s_i <1$
has the dimension of Hodge group equal $k$ if $(s_1,...,s_r)$
is in the face of quasiadjunction of depth $k$. The Milnor fiber
of the singularity $\Pi f_i^{a_i}$ is homotopy equivalent 
to the infinite cyclic cover of $S^3-L$ with the Galois group 
having the characters which when viewed as the characters of
$\pi_1(S^3-L)$ have the form $(t^{a_1},...,t^{a_r})$ 
for some $t$. By proposition (\ref{localsystemspectrum})
we can  identify an element of spectrum (with multiplicity)
in $(0,1)$ with the element of the universal cover of the torus 
of characters which exponent is the local system with non vanishing
$Gr_0^FH^1(S^3-L,L_{\xi})$ (of dimension equal to multiplicity).
Hence spectrum of $\Pi f_i^{a_i}$ consists of $s$ such that 
$(a_1s \ {\rm mod} 1, ....., a_rs \ {\rm mod} 1)$ belongs to a
face of quasiadjunction. This is equivalent to the description 
of the relation between the spectrum and the faces of quasiadjunction 
given in the proposition.

\begin{remark}  Similarly, the elements of the spectrum in $(-1,0)$
can be obtained as intersections of lines $(a_1s,...,a_rs)$ with 
conjugates of translates of faces of quasiadjunction. Indeed, the 
elements of the spectrum $(-1,0)$ characterized by the property that 
$exp(2 \pi \sqrt{-1} \alpha)$ are the eigenvalues on the semi-simple 
part of the monodromy acting on 
$Gr^W_1(F^0 \cap \bar F^1)H^1(M({a_1},...,{a_r}))$. 
On the other hand $Gr^W_1(F^0 \cap \bar F^1)$ is the conjugate of 
$Gr^W_1F^1 \cap \bar F^0=Gr^W_1Gr_F^1$ (since $H^1=F^0H^1, F^2H^1=0$).  
\end{remark}

\begin{remark}\label{alexanderpolynomial}
 If $\Pi (t_1^{m_1}...t_t^{m_r}-\omega)$ is the 
multivariable Alexander polynomial of $f_1...f_r$ then the 
Alexander polynomial of $\Pi f_i^{a_i}$ is given by 
$(t-1)\Pi(t^{a_1m_1+...+a_rm_r}-\omega)$. This follows from an 
extension of the classical relation between multivariable 
and one variable Alexander polynomial (cf. \cite{Turaev}).  

\end{remark}

In particular one obtains a relation between 
the faces of quasiadjunction of $f_1 \cdot \cdot \cdot f_r$
and the spectrum of corresponding singularity as follows.
\begin{cor} A rational $\xi \in (0,1)$ belongs to the spectrum of $f_1 \cdot ...
\cdot f_r$ if and 
only if 
\begin{equation}\label{equalxi}
(\xi,...,\xi) 
\end{equation}
belongs to a face of quasiadjunction.
Moreover the multiplicity of $\xi$ in the spectrum is 
the holomorphic depth of the character $\chi_{\gamma_i}$ given by 
$\chi_{\gamma_i}=e^{2 \pi \sqrt {-1} \xi}$.
In particular the number of faces of quasiadjunction 
intersecting the line in $\mathcal U$ given by points 
(\ref{equalxi}) is equal to the number of elements in the spectrum 
of $f_1 \cdot \cdot \cdot f_r$.
\end{cor}

The Milnor number of $f_1 \cdot \cdot \cdot f_r$ 
is related to the depths of faces of quasiadjunction 
as follow. If 
$dim Gr_0^WH^1(L_{\chi})=0$ for any $\chi$ which is the exponent 
of (\ref{equalxi}) then
\begin{equation}\label{milnornumber}
\mu=(r-1)+2\sum_F d^h(F) 
\end{equation}
where summation is over faces of quasiadjunction containing 
points of the form (\ref{equalxi})
and extra $r-1$ is due to the fact that this is the difference between the 
Betti number of 
finite branch and unbranched cover of $S^3-L$; the second part 
of (\ref{milnornumber}) calculates
the Betti number of branched covering of $S^3$. 

\bigskip

{\bf Example} The situation in which the number
 of faces of quasiadjunction 
 coincides with the number of elements in the spectrum comes up
in cases $(x^2+y^3)(y^2+x^3)$. 
For the singularity $x^r-y^r=0$ the sum of the depths of all faces is
${{(r-1)(r-2)} \over 2}$ and (**) translates to 
$r-1+2{{(r-1)(r-2)} \over 2}=(r-1)^2$

\bigskip\begin{remark} One may conjecture the following method 
to calculate faces of quasiadjunction. Let $exp: {\mathcal U} \rightarrow
Char^u H_1(S^3-Z(f_1,..f_r))$ be the universal covering map 
of the group of unitary characters restricted to the 
unit cube ${\mathcal U} \subset 
\R^r$. The subset of $\mathcal U$ which is the 
preimage of the zero set of the multivariable Alexander 
polynomial in $Char^u H_1(S^3-Z(f_1,..f_r))$ is a union of hyperplanes. 
Consider the stratification of this union in which each stratum 
consists of points belonging to the same collection of hyperplanes. 
The conjecture is that holomorphic depth is constant
on each stratum and hence the faces of quasiadjunction 
are the strata of this stratification with positive holomorphic depth.

In particular, in the case of two branches $V_1$ is a union of segments 
and $V_k$ with $k \ge 2$ are union of isolated points. The holomorphic 
depth of each point 
$(\alpha_1,\alpha_2)$ ($\alpha_i \in \Q$) 
(e.g. a component of the zero dimensional 
stratum of the preimage of the zero set of Alexander polynomial)
is the multiplicity of  $exp(2 \pi i t)$ of the singularity 
$f_1^{a_1}f_2^{a_2}=0$ 
where ${{\alpha_1} / {\alpha_2}}={{a_1} / {a_2}}$ (here 
$a_1, a_2$ are relatively prime integers and $t \in \Q$ is determined 
from $\alpha_i=a_i t$. 
\end{remark}

\section{Faces of quasiadjunction and 
Bernstein polynomials}\label{sectionbernstein}

Recall a definition of the multivariable Bernstein 
ideal following \cite{Sabbah} and \cite{Gyoja}. Consider the collection 
of polynomials $b(s_1,...,s_r) \in \C[s_1,...,s_r]$ and differential 
operators $P(s_1,..,s_r) \in {\mathcal D}[s_1,...,s_r]$ such that:

\begin{equation}
       P \cdot f_1^{s_1+1} \cdot ....\cdot f_r^{s_r+1}=
b(s_1,...s_r)f^{s_1} \cdot ...\cdot f_r^{s_r}
\end{equation}
where $f_1,...f_r$ are the germs of plane curves as in section
\ref{sectiononideals}. 
Such polynomials $b$ form an ideal $\mathcal B_{f_1,...,f_r}$ 
in $\C[s_1,...,s_r]$. This ideal is a non-zero ideal (cf. \cite{Sabbah}).
In the case $r=1$ the monic generator 
of it (uniquely defined) is called the Bernstein polynomial of germ 
$f$. The goal of this section is to show the following:

\begin{thm}\label{bernstein} Let $\mathcal P$ be the product of linear forms 
$L_i(s_1+1,...,s_r+1)$ where $L_i$ runs through linear forms  
vanishing on $(r-1)$-dimensional faces of 
of polytopes of quasiadjunction  
corresponding to a germ with 
$r$ irreducible components 
$f_1,...,f_r$. Then  
any $b \in \mathcal B_{f_1,...,f_r}$ 
is divisible by $\mathcal P$.
\end{thm}

\begin{proof} Let $L_{\phi}$ be the equation of a face of quasiadjunction 
corresponding to a germ $\phi \in \mathcal A$ 
and let $\omega_{\phi}$ be the form 
(\ref{phiform}). Using relation (\ref{zandf}) 
we have:

\begin{equation}\label{omegaphinew}
    \omega_{\phi}=\phi \cdot 
f_1^{{{j_1+1} \over m_1}-1}....f_r^{{{j_r+1} \over m_r}-1} dx \wedge dy
\end{equation}

This form extends over the exceptional locus of the resolution 
of singularity (\ref{zandf}) if and only if $\omega_\phi$ is $L^2$-form
(cf. \cite{merle}). It follows from the definition of an ideal of quasiadjunction 
that the integral of form (\ref{omegaphinew}) (over a small compact ball 
$B$ about the origin and a test function $\psi$ on $B$):
\begin{equation}\label{integralofomegaphi}
 \int_B \vert \phi \vert^2 
\vert f_1 \vert^{2({{j_1+1} \over {m_1}}-1)} 
\cdot ...\cdot \vert f_r \vert^{{2({{j_1+1} \over {m_1}}-1)}}
 \psi
\end{equation}
converges for all test functions $\psi$ and 
 all $\phi$ in the ideal of quasiadjunction corresponding 
to a face given by $L(x_1,...x_r)=0$ provided that $L({{j_1+1} \over {m_1}},...
{{j_r+1} \over {m_r}})>0$ and if $L({{j_1+1} \over {m_1}},...
{{j_1+r} \over {m_r}})=0$ then there is $\phi$ in this ideal of quasiadjunction 
for which the 
integral  diverges. 

On the other hand we have:

\begin{equation}\label{PPbar}
 \vert b(s_1,...,s_r) \vert^2 \vert \phi \vert^2 
\vert f_1 \vert^{2s_1} \cdot ...\cdot \vert f_r \vert^{2s_r} \psi=
P\bar P  {\vert f_1 \vert^2}^{s_1+1} \cdot ....\cdot 
{{\vert f_r \vert}^2}^{s_r+1} \vert \phi \vert^2 \psi
\end{equation}
since holomorphic and anti-holomorphic differential operators 
commute (as above, $\psi$ is a test function).
Hence integrating over $B$ both sides of (\ref{PPbar}) yields

\begin{equation}\label{equationwithb}
\vert b(s_1,...,s_r) \vert^2  \int_B \vert \phi \vert^2 
\vert f_1 \vert^{2s_1} \cdot ...\cdot \vert f_r \vert^{2s_r} \psi=
\end{equation}
$$\int P\bar P {\vert f_1 \vert^{2(s_1+1)}} \cdot ....\cdot 
{{\vert f_r \vert}}^{2(s_r+1)} \vert \phi \vert^2 \psi$$


Applying this to $\phi$ for which integral (\ref{integralofomegaphi})
diverges for $s_i={{j_i+1} \over {m_i}}-1$ where ${{j_i+1} \over {m_i}}$
is on the face of quasiadjunction we obtain that 
the meromorphic function of $(s_1,...,s_r)$ given by either side of  
(\ref{equationwithb})
when we approach to any point on the face of quasiadjunction 
right hand side is holomorphic in $s_1,..s_r$ while the integral 
in the left hand 
side has pole. Hence $b(s_1,...s_r)$ must vanish on $L(s_1+1,...,s_r+1)=0$.
\end{proof}

\begin{example} The hyperplanes containing the faces of quasiadjunction
of singularity $(x^2+y^3)(x^3+y^2)=0$ are $6x_1+4x_2=1,3,5$ and 
$4x_1+6x_2=1,3,5$ (cf. \cite{PartII} and \cite{Hodge}). Hence polynomials
in Bernstein ideal are vanishing on $6s_1+4s_2+k=0, 4s_1+6s_2+k=0, k=5,7,9$.
In fact according calculations of A.Leykin using SINGULAR 
(private communication) the Bernstein ideal for this singularity is 
principal with generator:
$$(s_1+1)(s_1+1)\Pi_{k=5,7,9,11,13} (4s_1+6s_2+k)(6s_1+4s_2+k)$$
\end{example}

\begin{remark} Similar to theorem \ref{bernstein} result, i.e. an 
identification of  
the faces of quasi-adjunction with the zeros of polynomials 
in the multivariable Bernstein ideal can be shown for 
isolated non normal crossibgs 
singularities discussed in \cite{INNC}. 
The details will appear in a forthcoming 
publication. 
\end{remark}

\end{document}